\documentclass[a4paper,10pt]{amsart}
\usepackage{amsmath}
\usepackage{amssymb, amsthm, xcolor,enumerate, mathtools}
\usepackage{verbatim}
\usepackage{enumitem}
\usepackage[hidelinks]{hyperref}
\usepackage{graphicx}
\usepackage{tikz, tikz-cd}
\usepackage{color}
\usepackage{xr}
\usepackage{booktabs}
\newtheorem{theorem}{Theorem}
\newtheorem{lemma}[theorem]{Lemma}
\newtheorem{proposition}[theorem]{Proposition}

\theoremstyle{definition}
\newtheorem{definition}[theorem]{Definition}

\newtheorem*{remark*}{Remark}

\numberwithin{equation}{section}
\numberwithin{theorem}{section}
\usepackage{thmtools}
\usepackage{thm-restate}
\makeatletter
\def\namedlabel#1#2{\begingroup
	#2
    \def\@currentlabel{#2}
    \label{#1}\endgroup
}
\makeatother
\newcommand{\M}{\mathcal{M}}
\newcommand{\e}{\epsilon}
\renewcommand{\d}{\delta}
\newcommand{\N}{\mathbb{N}}

\newcommand{\R}{\mathbb{R}}
\newcommand{\C}{\mathbb{C}}
\newcommand{\K}{\mathbb{K}}

\newcommand{\Z}{\mathcal{Z}}

\newcommand{\inv}{\mathrm{Inv}}
\newcommand{\Cu}{\mathrm{Cu}}

\newcommand{\Aff}{\mathrm{Aff}}

\let\oldtocsubsection=\tocsubsection
\renewcommand{\tocsubsection}[2]{\hspace{.75cm}\oldtocsubsection{#1}{#2}}
\DeclareRobustCommand{\SkipTocEntry}[5]{}
\usepackage{perpage}
\newcounter{mparcnt}
\MakePerPage{mparcnt}

\newcommand{\completion}[2]{\overline{#1}{}^{#2}}

\title[The real and stable rank of tracially complete C*-algebras]{The real and stable rank of tracially complete C*-algebras}
\author[]{Samuel Evington}
\address{Samuel Evington, Mathematical Institute, University of M\"unster, Einsteinstrasse 62, 48149 M\"unster, Germany}
\email{evington@uni-muenster.de}

\author[]{Aaron Tikuisis}
\address{Aaron Tikuisis, Department of Mathematics and Statistics, University of
  Ottawa, 585 King Edward, Ottawa, ON, K1N 6N5, Canada}
\email{aaron.tikuisis@uottawa.ca}

\thanks{This work was supported by:
Deutsche Forschungsgemeinschaft (DFG, German Research Foundation) – Project-ID 427320536 – SFB 1442 (Evington); 
Germany's Excellence Strategy EXC 2044/2 390685587  Mathematics M{\"u}nster: Dynamics–Geometry–Structure (Evington);  
ERC Advanced Grant 834267 - AMAREC (Evington);
NSERC Discovery Grant (Tikuisis), All Souls College visiting professorship (Tikuisis).
}

\begin{document}

\begin{abstract}
We prove that a factorial tracially complete C$^*$-algebra with CPoU has real rank zero and stable rank one. This leads to an essentially complete description of the Cuntz semigroup of these algebras. In particular, the results of this paper hold for the uniform tracial completions of $\mathcal{Z}$-stable C$^*$-algebras.
\end{abstract}

\maketitle

\section*{Introduction}
\renewcommand{\thetheorem}{\Alph{theorem}} 

The real rank of a C$^*$-algebra was introduced by Brown and Pedersen (\cite{BrownPed91}), based on the observation that a compact Hausdorff space $X$ has covering dimension at most $n$ if and only if every continuous function $f:X \rightarrow \R^{n+1}$ can be perturbed so that the range does not contain the origin. Brown and Pedersen's work builds on Rieffel's earlier notion of topological stable rank (\cite{Rieffel83}), which was later shown to coincide with the Bass stable rank (\cite{HV84}).

The minimal values of the real and stable rank are of particular interest as they relate to the relative abundance or scarcity of invertible elements in a given C$^*$-algebra $A$. 
Indeed, $A$ has \emph{real rank zero} if the set of invertible self-adjoint elements $\inv(A) \cap A_{sa}$ is dense in set of self-adjoint elements $A_{sa}$, and $A$ has \emph{stable rank one} if the set of invertible elements $\inv(A)$ is dense in $A$. 

Every von Neumann algebra has real rank zero (by Borel functional calculus), and every finite von Neumann algebra has stable rank one (by using unitary polar form). Moreover, there are close connections between these notations of rank and the Elliott classification programme for simple nuclear C$^*$-algebras (see the introduction to \cite{TraciallyComplete} for a survey). Indeed, every approximately finite dimensional algebra (\cite{BrownPed91}) and every Kirchberg algebra has real rank zero (\cite{Zhang90}), and every stably finite classifiable C$^*$-algebra has stable rank one (\cite{Ro04}). In particular every classifiable C$^*$-algebra has either real rank zero or stable rank one, and often both.

In this paper, we shall consider the real rank and the stable rank of factorial \emph{tracially complete C$^*$-algebras} (see Definition~\ref{def:TC} below). These C$^*$-algebras arise naturally when one completes the unit ball of a C$^*$-algebra $A$ with respect to the uniform 2-seminorm $\|a\|_{2,T(A)} = \sup_{\tau \in T(A)} \tau(a^*a)^{1/2}$. Here, $T(A)$ denotes the space of tracial states on $A$, which is assumed to be non-empty. Provided $T(A)$ is compact (which is automatic when $A$ is unital), the result of this construction is a C$^*$-algebra $\completion{A}{T(A)}$ with a bundle-like structure over $T(A)$, where the fibres are the GNS-representations $\pi_\tau$ with respect to each $\tau \in T(A)$. This viewpoint emerged from Ozawa's work on \emph{W$^*$-bundles} (\cite{Oz13}, see also \cite{Ev18}), and was subsequently generalised to arbitrary trace simplices by the authors along with Carri\'on, Castillejos, Gabe, Schafhauser, and White (\cite{TraciallyComplete}).

Tracially complete C$^*$-algebras provide a general framework for carrying out \emph{tracial gluing} arguments, which allow one to utilise properties of the von Neumann algebras $\{\pi_\tau(A)'': \tau \in T(A)\}$ to understand the structure of the tracial completion $\completion{A}{T(A)}$. Through techniques such as Matui and Sato's property (SI) (\cite{MS12,MS14}), one can in turn use the structure of $\completion{A}{T(A)}$ to prove results about the underlying C$^*$-algebra $A$. Tracial gluing arguments have underpinned recent progress on the Toms--Winter regularity conjecture (\cite{CETWW, CE, CETW}), the abstract approach to the C$^*$-classification (\cite{classification1,CETW-classification}), and dynamical notations of regularity (\cite{KLTV,SW25}).

The technical machinery for implementing tracial gluing arguments is known as \emph{complemented partitions of unity} (CPoU). Informally, a tracially complete C$^*$-algebra has CPoU provided there are (approximately) projection-valued partitions of unity over the trace space (see Definition~\ref{def:CPoU}). Under mild regularity conditions on $A$, which are satisfied when $A$ is classifiable in the sense of the Elliott classification programme, the uniform tracial completion of $A$ will have CPoU (\cite{CETWW, TraciallyComplete}). 

This brings us to the first main theorem of this paper.
\begin{theorem}\label{intro-thm:RR0}
    Let $(\M,X)$ be a factorial tracially complete C$^*$-algebra with CPoU. Then $\M$ has real rank zero.
\end{theorem} 
The conclusion that $\M$ has real rank zero is stronger than previously known results giving real rank zero of the tracial ultrapower $\M^\omega$ (\cite{TraciallyComplete, Fu26}). 
Indeed, to prove that $\M$ has real rank zero, one must achieve approximations in operator norm $\|\cdot\|$, whereas for proving that the tracial ultrapower $\M^\omega$ has real rank zero approximations with respect to the weaker uniform 2-norm $\|\cdot\|_{2,X}$ are sufficient. Accordingly, the proof of Theorem~\ref{intro-thm:RR0} is substantially harder than typical tracial gluing arguments based on CPoU (such as those in \cite[Section 7]{TraciallyComplete}).

Our approach is to show that every hereditary C$^*$-subalgebra $B$ of $\M$ contains a (not necessarily increasing) approximate unit consisting of projections. This property, first studied under the name property (HP) in \cite{Ped80}, was shown to be equivalent to real rank zero in Brown and Pedersen's original paper on real rank zero (\cite{BrownPed91}). In fact, it is enough to show that for any $x \in B$ and $\e > 0$ there is a projection $p \in B$ such that $\|px - x\| < \e$, using a standard argument invoking the C$^*$-identity (see \eqref{eqn:CuteTrick} below).  

The strategy is to construct a $\|\cdot\|_{2,X}$-Cauchy sequence of $\|\cdot\|_{2,X}$-approximate projections using CPoU, whose error tends to zero. By completeness, such a sequence must have a $\|\cdot\|_{2,X}$-norm limit in our tracially complete C$^*$-algebra $(\M,X)$, and it is not hard to show that the limit must be a projection. The difficulty is to ensure that this projection lies in a desired hereditary C$^*$-subalgebra $B$ (which is hard as $B$ need not be $\|\cdot\|_{2,X}$-closed) and ensure it acts like an approximate unit (which is also a $\|\cdot\|$-norm condition). 
The key idea is to use the ``acts as a unit'' relation $a \lhd b$ (see Section \ref{subsec:notation}) as an interlocutor between the $\|\cdot\|_{2,X}$-world of CPoU arguments and the $\|\cdot\|$-world of hereditary C$^*$-sublagberas. This builds on the techniques developed in \cite{Ev25} for constructing projections in hereditary subalgebras.

Motivated by this success, one might hope that similar techniques could be used to prove that factoriall tracially complete C$^*$-algebra with CPoU have stable rank one. 
This is indeed the case.
\begin{theorem}\label{intro-thm:SR1}
    Let $(\M,X)$ be a factorial tracially complete C$^*$-algebra with CPoU. Then $\M$ has stable rank one.
\end{theorem}
This time the focus is on unitaries, rather than on projections. More precisely, the aim is to find elements $a$ with a unitary polar form $a=u|a|$ for some unitary $u$. Any such element $a$ is then $\e$-close to the invertible element $u(|a|+\e)$.

In a tracial von Neumann algebra $(M,\tau)$, every element $a \in M$ has a (not necessarily unique) unitary polar form. However, there is no $\|\cdot\|_{2,\tau}$-continuous map $a \mapsto u(a)$ such that $a=u(a)|a|$ for all $a \in M$. We get around this by replacing $a$ with a $\|\cdot\|$-perturbation $a^{[\e]}$, defined using functional calculus (see Definition~\ref{def:perturbation}), before choosing a unitary polar form. This leads to $\|\cdot\|_{2,\tau}$-estimates that can be tracially glued together with CPoU. Unfortunately, this only results in a $\|\cdot\|_{2,X}$-approximation to $a^{[\e]}$ with a unitary polar form. However, an iterative version of this argument enables us to find an invertible element near $a$ using the fact that a $\|\cdot\|_{2,X}$-Cauchy sequence of unitaries in $\M$ converges to a unitary. 

As an application of Theorem~\ref{intro-thm:RR0}, we provide a new proof that the trace problem has a positive solution for type II$_1$ factorial tracially complete C$^*$-algebras with CPoU (first shown in \cite{Ev25}).
\begin{theorem}\label{intro-thm:TP}
      Let $(\M,X)$ be a type II$_1$ factorial tracially complete C$^*$-algebra with CPoU. Then $T(\M) = X$.
\end{theorem}
Indeed, since $\M$ has real rank zero, it suffices to show that every $\tau \in T(\M)$ agrees with some $\tau' \in X$ on all projections in $\M$. However, this is a consequence of Kadison duality (\cite{Kadison51}) since $K_0(\M)$ is isomorphic to the space of continuous affine functions $\Aff(X, \R)$.

Further, using Theorems~\ref{intro-thm:RR0} and \ref{intro-thm:SR1}, we provide a complete description of Cuntz subequivalence and the Cuntz semigroup of a type II$_1$ factorial tracially complete C$^*$-algebra with CPoU. The following theorem summaries the main observations (see Section~\ref{subsec:cuntz} for the relevant definitions).  
\begin{theorem}\label{intro-thm:Cuntz}
    Let $(\M,X)$ be a type II$_1$ factorial tracially complete C$^*$-algebra with CPoU. Then
    \begin{enumerate}
        \item For any, $a, b \in \M_+$, $a \precsim b$ if and only for every $\e> 0$ there exists a unitary $u \in \M$ such that $u(a-\epsilon)_+u^* \in \overline{bAb}$.
        \item The Cuntz semigroup $\Cu(\M)$ is algebraic and its compact part is isomorphic to $V(\M) \cong \Aff(X,\R^+_0)$.
        \item The Cuntz semigroup $\Cu(\M)$ is almost unperforated and almost divisible, i.e.\ $\M$ is pure.  
    \end{enumerate}
\end{theorem}
The first point of Theorem~\ref{intro-thm:Cuntz} is a consequence of stable rank one first observed by R{\o}rdam in \cite{Ro92}. It implies that the Cuntz semigroup has a weak form of cancellation that simplifies the theory substantially; see \cite{APLT22}. The second point of Theorem~\ref{intro-thm:Cuntz} is a consequence of real rank zero first observed in \cite{CEI08} and formalised in \cite[Section 5.5]{RFT18}. It follows that $\Cu(\M)$ is isomorphic to the Cu-completion of $\Aff(X,\R^+_0)$, i.e.\ the Cu-semigroup obtained from $\Aff(X,\R^+_0)$ by adding formal suprema of increasing sequences. This means that $\Cu(\M)$ can be completely understood by analysing the ordered semigroup $\Aff(X,\R^+_0)$. The proof of the third point of Theorem~\ref{intro-thm:Cuntz} is an example of this in action and provides a new class of pure C$^*$-algebras (\cite{Wi12, APTVarxiv}).

\renewcommand{\thetheorem}{\arabic{theorem}}
\numberwithin{theorem}{section}

\section{Preliminaries}\label{sec:prelims}

\subsection{Notation}\label{subsec:notation}
Let $A$ be a C$^*$-algebra. 
Write $A_+$ for the positive elements in $A$ and $A_{+,1}$ for the positive contractions. For $a,b \in A_+$, write $a \lhd b$ to mean $ab = a = ba$. For $0 \leq r < s \leq  1$, let $\eta_{r,s}:[0,\infty) \rightarrow \R$ be the continuous function defined by
    \begin{equation}\label{def:eta}
        \eta_{r,s}(x) = 
        \begin{cases} 
        0, &0 \leq x \leq r,\\
        \mathrm{affine},  &r \leq x \leq s,\\
        1, &s \leq x < \infty.\\
        \end{cases}
    \end{equation} 
Note that $\eta_{r,s}(a) \lhd \eta_{r',s'}(a)$ if $0 \leq r' < s' \leq r < s < t \leq  1$ and $a \in A_{+}$.

\subsection{Uniform 2-seminorms}
Let $A$ be a C$^*$-algebra. Let $T(A)$ denote the space of tracial states on $A$.
For each $\tau \in T(A)$, the corresponding \emph{2-seminorm} is given by $\|a\|_{2,\tau} = \tau(a^*a)^{1/2}$ for all $a \in A$.
For every non-empty compact convex subset $X \subseteq T(A)$, we define the \emph{uniform 2-seminorm} with respect to $X$ by
\begin{equation}
    \|a\|_{2,X} = \sup_{\tau \in X} \|a\|_{2,\tau}
\end{equation}
for $a \in A$. Although $\|\cdot\|_{2,X}$ is not sub-multiplicative in general, we have 
\begin{equation}
    \|ab\|_{2,X} \leq \max(\|a\|\|b\|_{2,X}, \|a\|_{2,X}\|b\|)
\end{equation}
for all $a,b \in A$. Therefore, multiplication in $A$ is $\|\cdot\|_{2,X}$-continuous, when restricted to $\|\cdot\|$-bounded sets. For more information on uniform 2-seminorms, see \cite[Section~3.1]{TraciallyComplete}.

\subsection{Tracially complete C*-algebras}
We now recall the definition of a tracially complete C$^*$-algebra.
\begin{definition}[{\cite[Definition~3.4]{TraciallyComplete}}]\label{def:TC}
	A \emph{tracially complete C$^*$-algebra} is a pair $(\M,X)$ where $\M$ is a unital C$^*$-algebra and $X \subseteq T(\M)$ is a non-empty compact convex set such that
	\begin{enumerate}
		\item $\|\cdot\|_{2,X}$ is a norm on $\M$, and
		\item the $\|\cdot\|$-closed unit ball of $\M$ is $\|\cdot\|_{2,X}$-complete.
	\end{enumerate}
\end{definition}
A tracially complete C$^*$-algebra $(\M, X)$ is \emph{factorial} if $X$ is a face in $T(\M)$. For $\tau \in X$, let $\pi_\tau:\M \rightarrow B(L^2(\M,\tau))$ be the induced GNS-representation. We say $(\M, X)$ is \emph{type} II$_1$  if the von Neumann algebra $\pi_\tau(\M)''$ is type II$_1$ for all $\tau \in X$.

\subsection{Uniform tracial completions}
The motivating examples of tracial complete C$^*$-algebras are uniform tracially completions of C$^*$-algebras. Let $A$ be a C$^*$-algebra. For a non-empty compact convex set $X \subseteq T(A)$, the \emph{uniform tracial completion} of $A$ with respect to $X$ is the C$^*$-algebra
\begin{equation}
	\completion{A}{X} = \frac{\{(a_n)_{n=1}^\infty \in \ell^\infty(A): (a_n)_{n=1}^\infty \text{ is }\|\cdot\|_{2,X}\text{-Cauchy}\}}{\{(a_n)_{n=1}^\infty \in \ell^\infty(A): (a_n)_{n=1}^\infty \text{ is }\|\cdot\|_{2,X}\text{-null}\}}.
\end{equation}
Every $\tau \in X$ induces a trace $\hat{\tau} \in T(\completion{A}{X})$ via $\hat{\tau}((a_n)_{n=1}^\infty) = \lim_{n\rightarrow\infty} \tau(a_n)$. Hence, we may identify $X$ with a subset of $\hat{X} \subseteq T(\completion{A}{X})$. The pair $(\completion{A}{X},\hat{X})$ is a tracially complete C$^*$-algebra and is factorial whenever $X$ is a face in $T(A)$; see \cite[Section~3.3]{TraciallyComplete}.

\subsection{Ultrapowers of tracially complete C*-algebras}
Let $\omega \in \beta\N \setminus \N$ be a free ultrafilter. The \emph{ultrapower} $(\M^\omega,X^\omega)$ of the tracially complete C$^*$-algebra $(\M,X)$ with respect to $\omega$ is defined as
\begin{equation}
    \M^\omega = \frac{\ell^\infty(\M)}{\{(a_n)_{n=1}^\infty \in \ell^\infty(A): \lim_{n\to\omega}\|a_n\|_{2,X} = 0\}}.
\end{equation}
For every sequence of traces $(\tau_n)_{n=1}^\infty$ in $X$, we define a \emph{limit trace} on $\M^\omega$ via $(a_n)_{n=1}^\infty \mapsto \lim_{n\to\omega} \tau_n(a_n)$, and we set $X^\omega \subseteq T(\M)$ to be the weak$^*$-closure of the set of all limit traces. For further information, see \cite[Section 5.1]{TraciallyComplete}.  

\subsection{Complemented partitions of unity (CPoU)}

We now recall the formal definition of complemented partitions of unity (CPoU). 
\begin{definition}[{\cite[Definition~6.1]{TraciallyComplete}}]\label{def:CPoU}
	Let $(\M,X)$ be a factorial tracially complete C$^*$-algebra. We say that $(\M,X)$ has \emph{complemented partitions of unity} (CPoU) if for any $\|\cdot\|_{2,X}$-separable subset $S \subseteq \M$, any family $a_1,\dots,a_n$ of positive elements in $\M$, and any scalar
	\begin{equation}\label{eq:CPoUTraceIneq1}
		\delta>\sup_{\tau \in X} \min_{1 \leq j \leq n}\tau(a_j),
	\end{equation}
	there exist orthogonal projections $e_1,\dots,e_n\in \M^\omega\cap S'$ summing to $1_{\M^\omega}$ such that
	\begin{equation}\label{eq:CPoUTraceIneq2}
		\tau(a_je_j)\leq \delta\tau(e_j)
	\end{equation}
        for all $\tau\in X^\omega$ and $j=1,\dots,n$.
\end{definition}
In this definition, $\omega$ can be taken to be any free ultrafilter on $\N$. For  
a proof that the definition is independent of the choice of ultrafilter, see \cite[Proposition~6.2]{TraciallyComplete}.
By \cite[Theorem 1.4]{TraciallyComplete}, the uniform tracial completion $(\completion{A}{X},\hat{X})$ of a C$^*$-algebra $A$ with respect to a non-empty compact face $X \subseteq T(A)$ has CPoU whenever $A$ is $\Z$-stable (or more generally when $A$ has uniform property $\Gamma$).
 We refer the reader to \cite[Section~7]{TraciallyComplete} for further information and examples of CPoU in action.

\subsection{The Cuntz semigroup}\label{subsec:cuntz}

In this subsection, we recall the background on the Cuntz semigroup required to understand the statement and proof of Theorem~\ref{intro-thm:Cuntz}. It is not used elsewhere in this paper. 

Let $A$ be a C$^*$-algebra and let $a,b \in (A \otimes \K)_+$. The Cuntz subequivalence relation $a \precsim b$ is defined to hold precisely when there exists a sequence $(r_n)_{n=1}^\infty$ in $A \otimes \K$ such that $a = \lim_{n\to\infty} r_nbr_n^*$.  Projections are Cuntz subequivalent if and only if they are Murray von Neumann subequivalent by \cite[Proposition 2.1]{Ro92}.
The Cuntz equivalence relation $a \sim b$ is defined to hold whenever $a \precsim b$ and $b \precsim a$.  When $A$ is stably finite, Cuntz equivalence is the same as Murray--von Neumann equivalence. 

The Cuntz semigroup is defined to be the quotient $\Cu(A) = (A \otimes \K)_+ / \sim$. The semigroup operation is induced by the direct sum using that $M_2(\K) \cong \K$. Cuntz subequivalence descends to an order on the semigroup $\Cu(A)$. 

Every increasing sequence $(x_n)_{n=1}^\infty$ in $\Cu(A)$ has a supremum $\sup_n x_n$ by \cite{CEI08}. For $x, y \in \Cu(A)$, the relation $x \ll y$ holds if for every sequence $(y_n)_{n=1}^\infty$ with $y \leq \sup_n y_n$ there exits $n \in \N$ such that $x \leq y_n$. An element $x \in \Cu(A)$ is said to be \emph{compact} if $x \ll x$.
For $a \in (A \otimes \K)_+$, we have $[(a-\epsilon)_+] \ll [a]$, for all $\e > 0$, 
and $[a]$ is compact whenever $a$ is Cuntz equivalent to a projection. If $A$ is stably finite, the converse holds and we can identify the subsemigroup of compact elements of $\Cu(\M)$ with the Murray--von Neumann semigroup $V(\M)$; see \cite[Proposition 3.5]{BrownCiuperca08}.

A Cuntz semigroup is \emph{algebraic} if every element element is the supremum of a increasing sequence of compact elements. In this case, the Cuntz semigroup can be recovered from the subsemigroup $S_c \subseteq \Cu(A)$ of compact elements via Cu-completion; see \cite[Section 5.5]{RFT18}.  A Cuntz semigroup $\Cu(A)$ is \emph{almost unperforated} if $x \leq y$ whenever $(k+1)x \leq ky$ for some $k \in \N_+$. This is well-known to be equivalent to $A$ having strict comparison; see for example \cite[Proposition 3.2]{Ro04}. A Cuntz semigroup $\Cu(A)$ is \emph{almost divisible} if for any $x,x' \in \Cu(A)$ such that $x' \ll x$ and any $k \in \N_+$ there exists $y \in \Cu(A)$ with $ky \leq x$ and $x' \leq (k+1)y$. If $\Cu(A)$ is both almost unperforated and almost divisible, then $A$ is \emph{pure} in the sense of \cite{Wi12}.

\section{Real rank zero}\label{sec:RRO}

In this section, we prove Theorem~\ref{intro-thm:RR0}. We begin with a von Neumann algebraic result that we will apply to each fibre of a tracially complete C$^*$-algebra. In essence, this is a well-known result on 2-norm stability of projections. We simply observe that the usual proof interacts nicely with the ``acts as a unit'' relation.   
 
\begin{lemma}\label{lem:vNA-stab-projection}
    Let $M$ be a finite von Neumann algebra with trace $\tau$. Let $0 < \delta \leq \tfrac{1}{4}$ and $a,p,c \in M_{+,1}$. Suppose  $a \lhd p \lhd c$ and $\|p^2-p\|_{2,\tau} \leq \d$.
    Then there exists a projection $q \in M$ such that $a \lhd q \lhd c$ and
    \begin{equation}
        \|q-p\|_{2,\tau} \leq 2\sqrt{\d}.
    \end{equation}
\end{lemma}
\begin{proof}
    Set $q = \chi_{[1-\sqrt{\d},1]}(p)$. By \cite[Chapter XIV, Lemma~2.2]{Takesaki3}, $\|q-p\|_{2,\tau} \leq 2\sqrt{\d}$.\footnote{The estimate in the proof of \cite[Chapter XIV, Lemma~2.2]{Takesaki3} is still valid when $\delta = \tfrac{1}{4}$.}
    Since $p \lhd c$ and $q \in \mathrm{W}^*(p)$, we have $q \lhd c$. Since $a \lhd p$ and $\chi_{[1-\sqrt{\d},1]}(1) = 1$, it follows that $a \lhd q$.
\end{proof}

Next, we perform a standard tracial gluing argument using CPoU to the result of Lemma~\ref{lem:vNA-stab-projection}. The reader can find similar arguments in \cite[Section~7]{TraciallyComplete}. We note that the constant in the estimate of $\|p-q\|$ gets slightly worse, but for our argument the precise value of the constant doesn't matter.

\begin{proposition}\label{prop:RR0-CPoU-1}
    Let $(\M,X)$ be a factorial tracially complete $C^*$-algebra with CPoU.
    Let $\e > 0$ and let $a,p,c \in \M_{+,1}$ with  $a \lhd p \lhd c$. Suppose $\|p^2-p\|_{2,X} \leq \d$ for some $\d \in (0,\tfrac{1}{4}]$.
    Then there exists $q \in \M_{+,1}$ such that
 \begin{equation} \label{eqn:conditions-on-q}
        \|cq-q\|_{2,X} < \e, \quad
        \|qa-a\|_{2,X} < \e, \quad
        \|q^2-q\|_{2,X} < \e, \quad
        \|p-q\|_{2,X} < 4\sqrt{\d}.
 \end{equation}
\end{proposition}
\begin{proof}
Let $\tau \in X$. By Lemma~\ref{lem:vNA-stab-projection}, there exists a projection $\bar{q}_\tau \in \pi_\tau(\M)''$ such that $\pi_\tau(a) \lhd \bar{q}_\tau \lhd \pi_\tau(c)$ and $\|\bar{q}_\tau-\pi_\tau(p)\|_{2,\tau} \leq 2\sqrt{\d}$. Applying Kaplansky's density theorem and lifting the resulting positive contraction, we may find $q_\tau \in \M_{+,1}$ such that 
 \begin{equation} 
        \|cq_\tau-q_\tau\|_{2,\tau} < \e, \quad
        \|q_\tau-a\|_{2,\tau} < \e, \quad
        \|q_\tau^2-q_\tau\|_{2,\tau} < \e, \quad
        \|p-q_\tau\|_{2,\tau} < \sqrt{5\d}.
 \end{equation}
By continuity, there is an open neighbourhood $U_\tau$ of $\tau$ in X such that for all $\sigma \in U_\tau$
 \begin{equation} 
        \|cq_\tau-q_\tau\|_{2,\sigma} < \e, \quad
        \|q_\tau-a\|_{2,\sigma} < \e, \quad
        \|q_\tau^2-q_\tau\|_{2,\sigma} < \e, \quad
        \|p-q_\tau\|_{2,\sigma} < \sqrt{5\d}.
 \end{equation}
As $X$ is compact, there exist $n \in \N$ and positive contractions $q_1,\dots,q_n \in \M_{+,1}$  such that, for all $\tau \in X$, there exists $j$ such that
 \begin{equation} \begin{split}
    \|cq_j-q_j\|_{2,\tau}  < \tfrac{\e}{4}, \,
    \|q_ja-a\|_{2,\tau}  < \tfrac{\e}{4}, \,
    \|q_j^2-q_j\|_{2,\tau}  < \tfrac{\e}{4}, \,
    \|p-q_j\|_{2,\tau}  < \sqrt{5\d}.
\end{split} \end{equation}
Define 
\begin{equation} 
b_j= |cq_j-q_j|^2 +  |q_ja-a|^2 + |q_j^2-q_j|^2  + \frac{\e^2}{20\d}|p-q_j|^2, 
\end{equation}
so that
\begin{equation} \sup_{\tau \in X} \min_{j\in\{1,\dots,n\}} \tau(b_j) < \tfrac{1}{2}\e^2. \end{equation}
Applying CPoU, we obtain projections $e_1,\dots,e_n \in \M^\omega \cap \M'$ summing to $1_{\M^\omega}$, such that $\tau(e_jb_j)\leq \tfrac{1}{2}\e^2\tau(e_j)$ for all $j=1,\dots,n$ and all $\tau \in X^\omega$.

Set $h= \sum_{j=1}^n q_je_j \in \M^\omega$ and $b = |ch-h|^2 +  |ha-a|^2 + \d|h^2-h|^2  + \frac{\e^2}{20\d}|p-h|^2$.
Then for $\tau \in X^\omega$, using centrality and orthogonality of the $e_j$, we have
\begin{equation}\label{eqn:cpou-result-sum}
 \tau(b) = \tau\left(\sum_{i=1}^n e_jb_j\right) \leq \sum_{i=1}^n \frac12 \e^2\tau(e_j) = \frac{1}{2}\e^2
 \end{equation}
Since $|ch-h|^2 \leq b$, it follows that $\|ch-h\|_{2,X^\omega}^2 \leq \frac{1}{2}\e^2$. 
Similarly, we deduce from \eqref{eqn:cpou-result-sum} that $\|ha-a\|_{2,X^\omega}^2 \leq \frac{1}{2}\e^2$, $\|h^2-h\|_{2,X^\omega}^2 \leq \frac{1}{2}\e^2$ and $\|p-h\|_{2,X^\omega}^2 \leq 10\d$. 
Lifting $h \in \M^\omega$ to a sequence of positive contractions in $\M$, it follows that there exists an element $q \in \M_{+,1}$ in in this sequence that satisfies \eqref{eqn:conditions-on-q}.
\end{proof}

Now, we improve on the outcome of the previous proposition by showing that the first two estimates in \eqref{eqn:conditions-on-q} can be obtained exactly. The price we pay is a second layer of ``$\lhd$-bounds''.

\begin{proposition}\label{prop:RR0-CPoU-2}
    Let $(\M,X)$ be a factorial tracially complete $C^*$-algebra with CPoU.
    Let $\e > 0$ and let $a',a,p,c,c' \in \M_{+,1}$ with  $a' \lhd a \lhd p \lhd c \lhd c'$.
    Suppose $\|p^2-p\|_{2,X} \leq \d$ for some $\d \in (0,\tfrac{1}{4}]$.
    Then there exists $q \in \M_{+,1}$ such that $a' \lhd q \lhd c'$
    \begin{align}
        \|q^2-q\|_{2,X} \leq \e,\quad
        \|q-p\|_{2,X} \leq 4\sqrt{\d}.
    \end{align}
\end{proposition}
\begin{proof}
    Let $\e_0 > 0$. By Proposition~\ref{prop:RR0-CPoU-1}, there exists an element $q_0 \in \M_{+,1}$ such that
    $\|cq_0-q_0\|_{2,X} \leq \e_0$,
    $\|q_0a-a\|_{2,X} \leq \e_0$, 
    $\|q_0^2-q_0\|_{2,X} \leq \e_0$, and
    $\|p-q_0\|_{2,X} < 4\sqrt{\d}$. 
    
    Set $q = c((1-a)q_0(1-a)+a)c = c(aq_0a-aq_0-q_0a+q_0+a)c$. 
    Then $q \in \M_{+,1}$.
    Since $c \lhd c'$, it is immediate that $q \lhd c'$.
    Since $a' \lhd a \lhd c$, we compute that 
    \begin{equation}
        a'q = (a'q_0a - a'q_0 - a'q_0a + a'q_0 + a')c = a'c = a'.
    \end{equation}
    Hence, $a' \lhd q$.
    Since $\|q_0a-q_0\|_{2,X} \leq \e_0$ and $\|cq_0-q_0\|_{2,X} \leq \e_0$, we obtain the following chain of $\|\cdot\|_{2,X}$-approximations:
    \begin{equation}
        q \approx_{4\e_0} c(a-a-a+q_0+a)c = cq_0c \approx_{2\e_0} q_0.
    \end{equation}
    Hence, $\|q-q_0\|_{2,X} \leq 6\e_0$. From this and $\|q_0^2-q_0\|_{2,X} \leq \e_0$, we obtain the following chain of $\|\cdot\|_{2,X}$-approximations: 
    \begin{equation}
        q^2 \approx_{6\e_0} qq_0 \approx_{6\e_0} q_0^2 \approx_{\e_0} q_0 \approx_{6\e_0} q_0.
    \end{equation}    
    Hence, $\|q^2-q\| \leq 19\e_0$. Finally, we have $\|p-q\|_{2,X} \leq \|p-q_0\|_{2,X} + 6\e_0$ and $\|p-q_0\|_{2,X} < 4\sqrt{\d}$. 
    Taking $\e_0$ sufficiently small, the result follows.   
\end{proof}

The following theorem is the key construction. We inductively build a $\|\cdot\|_{2,X}$-Cauchy sequence of approximate projections using Proposition~\ref{prop:RR0-CPoU-2}, which must have a $\|\cdot\|_{2,X}$-limit. Moreover, we maintain some control over this $\|\cdot\|_{2,X}$-limit using a bi-infinite sequence of ``$\lhd$-bounds''.  
\begin{theorem}\label{thm:projection}
  Let $(\M,X)$ be a factorial tracially complete $C^*$-algebra with CPoU.
  Let $a,b,c \in \M_{+,1}$ with $a \lhd b \lhd c$. 
  Then there exists a projection $p \in \M$ with $a \lhd p \lhd c$.
\end{theorem}
\begin{proof}
    Let $\eta_{r,s}$ be the continuous piecewise affine functions defined in  \eqref{def:eta}.
    Let $(\alpha_n)_{n=0}^\infty$ be an increasing sequence of real numbers in the interval $(\tfrac{1}{2},1)$, and let $(\gamma_n)_{n=0}^\infty$ be a decreasing sequence of real numbers in the interval $(0,\tfrac{1}{4})$.
    Set $a_n = \eta_{\alpha_n, \alpha_{n+1}}(b)$, $b_0 = \eta_{1/4,1/2}(b)$ and $c_n = \eta_{\gamma_{n+1},\gamma_n}(b)$. 
    Then  
    \begin{equation}\label{eqn:sequences}
        a \lhd a_{n+1} \lhd a_n \lhd b_0 \lhd c_{n} \lhd c_{n+1} \lhd c
    \end{equation}
    for all $n \in \N_{\geq 0}$.
    Set $\e_n = \tfrac{1}{16^n}$ for $n \in \N$.
    By Proposition~\ref{prop:RR0-CPoU-2} with $p=b_0$ and $\delta=\|b_0^2-b_0\|_{2,X} \leq \frac{1}{4}$ (since $b_0$ is a contraction),\footnote{Of course, one can also prove an (easier) version of Proposition~\ref{prop:RR0-CPoU-2} that doesn't keep track of the distance between $p$ and $q$.} 
    we obtain $p_1 \in \M_{+,1}$ such that
    \begin{equation}
     a_1\lhd p_1 \lhd c_1\quad \text{and} \quad \|p_1^2-p_1\|_{2,X} \leq \e_1. 
    \end{equation}
    Then, by repeated application of Proposition~\ref{prop:RR0-CPoU-2}, we iteratively construct positive contractions $p_2,p_3,\dots$ in $\M$ such that for all $n \geq 2$,
    \begin{align}
        a_n  \lhd p_n &\lhd c_n, \label{eqn:p1}\\
        \|p_n^2 - p_n\|_{2,X} &\leq \e_n, \label{eqn:p2} \\
        \|p_{n-1}-p_{n}\|_{2,X} &\leq 4\sqrt{\e_{n-1}}.\label{eqn:p3}
    \end{align}
    
    Since $\sum_{n=2}^\infty 4\sqrt{\e_{n-1}} < \infty$, it follows from \eqref{eqn:p3} that the sequence $(p_n)_{n=1}^\infty$ is $\|\cdot\|_{2,X}$-Cauchy. As $(\M,X)$ is tracially complete, $(p_n)_{n=1}^\infty$ converges in $\|\cdot\|_{2,X}$-norm to a contraction $p \in \M$.
    Since multiplication of contractions is $\|\cdot\|_{2,X}$-continuous, it follows from \eqref{eqn:p2} that $p$ is a projection.
    From \eqref{eqn:sequences} and \eqref{eqn:p1}, we deduce that $ap_n = a$ and $p_nc = p_n$ for all $n \in \N$. Then, again using that multiplication of contractions is $\|\cdot\|_{2,X}$-continuous, we deduce that $a \lhd p \lhd c$ as required. 
\end{proof}

Finally, we obtain real rank zero, by showing that every hereditary subalgebra has a (not necessarily increasing) approximate unit of projections.
\begin{theorem}[Theorem~\ref{intro-thm:RR0}]
    Let $(\M,X)$ be a factorial tracially complete $C^*$-algebra with CPoU.
    Then $\M$ has real rank zero. 
\end{theorem}
\begin{proof}
    Let $B$ be a hereditary subalgebra of $\M$.
    Let $x_1,\ldots,x_n \in B$ and $\e > 0$. Set $x = \sum_{i=1}^n x_i^*x_i$.  
    Let $\eta_{r,s}$ be the continuous piecewise affine functions defined in  \eqref{def:eta}.
    Set $\delta = \tfrac{\e}{6}$. 
    Define $a = \eta_{2\delta, 3\delta}(x)$, $b = \eta_{\delta, 2\delta}(x)$, and $c = \eta_{0,\delta}(x)$. Then $a,b,c \in B$ and $a \lhd b \lhd c$.
    Moreover, we have $\|ax-x\| \leq 3\delta$ by functional calculus.
    
    By Theorem~\ref{thm:projection}, there exists a projection $p \in \M$ with $a \lhd p \lhd c$. Since $p \lhd c \in B$ and $B$ is hereditary, it follows that $p \in B$. Using $a \lhd p$ and the triangle inequality, we have
    \begin{equation}
    \begin{split}
        \|x-px\| &\leq \|x-ax\| + \|ax - pax\| + \|p(ax - x)\|\\
            &\leq 3\delta + 0 + 3\delta\\
            &=\e.
    \end{split}
    \end{equation}
    Using the C$^*$-identity, it follows that
    \begin{equation}\label{eqn:CuteTrick}
        \|(1-p)x_i\|^2 = \|(1-p)x_i^*x_i(1-p)\| \leq  \|(1-p)x^*x(1-p)\| = \|(1-p)x\|^2 < \e^2
    \end{equation} 
    for $i=1,\ldots,n$. It follows that $B$ has a (not necessarily increasing) approximate unit of projections. Therefore, $\M$ has real rank zero by \cite[Theorem 2.6]{BrownPed91}.    
\end{proof}

\section{Stable rank one}\label{sec:SR1}

In this section, we shall prove Theorem~\ref{intro-thm:SR1}. The following piece of functional calculus will play a crucial role in our computations, so we introduce a shorthand notation.
\begin{definition}\label{def:perturbation}
Let $A$ be a $C^*$-algebra, $a \in A$ and $\e>0$. Define
\begin{equation} a^{[\e]}= v(|a|-\e)_+, \end{equation}
where $a=v|a|$ is the polar decomposition in the bidual $A^{**}$. Note that $a^{[\e]}$ is a well-defined element of $A$, as it can be written as $af(|a|)$ for a suitable function $f\in C_0((0,\|a\|])_+$. Moreover, $\|a-a^{[\e]}\|\leq \e$.
\end{definition}

Our first lemma is a quantitative version of the well-known result that every element of a finite von Neumann algebra has a unitary polar form. We need to cut down from $a$ to $a^{[\e]}$ because a small perturbation to $a$ that is ``near $0$ in the spectrum'' could get amplified to a large change in the unitary of the polar decomposition.

\begin{lemma}
\label{lem:ControlledPolar}
Let $M$ be a finite von Neumann algebra with a trace $\tau$. Let $\e > 0$.
If $a \in M$ satisfies $\big\|a-|a|\,\big\|_{2,\tau} < \e^2$, then there exists a unitary $u \in M$ such that 
\begin{equation}
\|u-1\|_{2,\tau}<4\e\quad \text{and}\quad 
 a^{[\e]}=u|a^{[\e]}|.
\end{equation}
\end{lemma}

\begin{proof}
Let $a \in M$. Let $a=v|a|$ be the polar decomposition of $a$ in $M$, which is also the polar decomposition of $a$ in the bidual $M^{**}$ by uniqueness.  
Set
\begin{equation} p= \chi_{[\e,1]}(|a|),\quad q= \chi_{[\e,1]}(|a^*|) = vpv^*. \end{equation}
Then $p \leq \frac1{\e^2} |a|^2$, so that
\begin{equation}\begin{split}
\ \|(v-1)p\|_{2,\tau}^2 &= \tau\big((v-1)p(v^*-1)\big) \\
&\leq \frac1{\e^2} \tau\big((v-1)|a|^2(v^*-1)\big) \\
&= \frac1{\e^2} \tau\big((a-|a|)(a^*-|a|)\big) \\
&= \frac1{\e^2}\big\|a-|a|\big\|_{2,\tau}^2 \\
&< \frac1{\e^2}(\e^2)^2 = \e^2.
\end{split}\end{equation}
Therefore,
\begin{equation}\begin{split}
 \|p-q\|_{2,\tau} 
&= \|p-vpv^*\|_{2,\tau} \\
&\leq \|p(1-v^*)\|_{2,\tau}+\|(1-v)p\|_{2,\tau}\|v^*\| \\
&<\e+\e=2\e.
\end{split}\end{equation}
By \cite[Chapter XIV, Lemma 2.1]{Takesaki3}, there exists a unitary $w \in M$ such that $wpw^*=q$ and
\begin{equation}
\|w-1\|_{2,\tau} \leq \sqrt2\|p-q\|_{2,\tau} < 2\sqrt2\e.
\end{equation}
Set 
\begin{equation} u= vp + w(1-p) = qv+(1-q)w. \end{equation}
Then
\begin{equation} uu^* = (vp+w(1-p))(pv^*+(1-p)w^*) = vpv^*+w(1-p)w^* = q+(1-q) = 1, \end{equation}
and likewise $uu^*=1$, so that $u$ is unitary.
We also have
\begin{equation} \|u-1\|_{2,\tau} \leq \|(v-1)p\|_{2,\tau} + \|(w-1)(1-p)\|_{2,\tau} < \e+2\sqrt2\e < 4\e \end{equation}
and, since $p|a^{[\e]}|=|a^{[\e]}|=(|a|-\e)_+$, we have
\begin{equation} u|a^{[\e]}| = v(|a|-\e)_+ = a^{[\e]}. \qedhere \end{equation}
\end{proof}

Now, we perform a standard tracial gluing argument using CPoU to the result of Lemma~\ref{lem:ControlledPolar}, similar to the proof of Proposition~\ref{prop:RR0-CPoU-1} and arguments in \cite[Section~7]{TraciallyComplete}. 

\begin{proposition}
\label{prop:ControlledPolarCPoU}
Let $(\M,X)$ be a factorial tracially complete $C^*$-algebra with CPoU. Let $\e_0,\e_1 > 0$.
Suppose $a \in \M$ satisfies $\big\|a-|a|\,\big\|_{2,X} < \e_0^2$. Then there exists a unitary $u \in \M$ such that 
\begin{equation}
\|u-1\|_{2,X} \leq 5\e_0, \quad \quad \big\|a^{[\e_0]}- u|a^{[\e_0]}|\big\|_{2,X} \leq \e_1.
\end{equation}
\end{proposition}

\begin{proof}
Let $\tau \in X$. By Lemma~\ref{lem:ControlledPolar}, there exist a unitary $\bar{u}_\tau \in \pi_\tau(\M)''$ such that 
$\|\bar{u}_\tau-1\|_{2,\tau}<4\e$ and $\pi_\tau(a^{[\e]})=\bar{u}_\tau|\pi_\tau(a^{[\e]})|$. Since $\pi_\tau(\M)''$ is a von Neumann algebra, there is  $\bar{h}_\tau \in \pi_\tau(\M)_{sa}''$ such that $\bar{u}_\tau = e^{i\bar{h}_\tau}$. Using the Kaplansky density theorem and lifting the resulting self-adjoint, we may find $h_\tau \in \M_{sa}$ such that 
\begin{equation}
\|e^{ih_\tau}-1\|_{2,\tau} < 4\e_0, \quad \quad \big\|a^{[\e_0]}- e^{ih_\tau}|a^{[\e_0]}|\big\|_{2,\tau} < \tfrac{1}{2}\e_1.
\end{equation}
By continuity, there is an open neighbourhood $U_\tau$ of $\tau$ in X such that for all $\sigma \in U_\tau$
\begin{equation}
\|e^{ih_\tau}-1\|_{2,\sigma} < 4\e_0, \quad \quad \big\|a^{[\e_0]}- e^{ih_\tau}|a^{[\e_0]}|\big\|_{2,\sigma} < \tfrac{1}{2}\e_1.
\end{equation}
As $X$ is compact, there exist $n \in \N$ and self-adjoints $h_1,\dots,h_n \in \M$ such that, for all $\tau \in X$, there exists $j$ such that
\begin{equation} 
\|e^{ih_j}-1\|_{2,\tau} < 4\e_0, \quad \quad \big\|a^{[\e_0]}-e^{ih_j}|a^{[\e_0]}|\big\|_{2,\tau} <\tfrac{1}{2}\e_1.
\end{equation}
Define 
\begin{equation} b_j= \e_1^2|e^{ih_j}-1|^2 + 24\e_0^2\big|a^{[\e_0]}-e^{ih_j}|a^{[\e_0]}|\big|^2, \end{equation}
so that
\begin{equation} \sup_{\tau \in X} \min_{j\in\{1,\dots,n\}} \tau(b_j) < 16\e_1^2\e_0^2 + 6\e_0^2\e_1^2 = 22\e_1^2\e_0^2. \end{equation}
Applying CPoU with $\delta= 22\e_1^2\e_0^2$, we obtain projections $p_1,\dots,p_n \in \M^\omega \cap \M'$ summing to $1_{\M^\omega}$, such that
\begin{equation} \tau(p_jb_j)\leq 22\e_1^2\e_0^2\tau(p_j), \end{equation}
for all $i=1,\dots,n$ and all $\tau \in X^\omega$.
Set
\begin{equation} h= \sum_{j=1}^n h_jp_j \in \M^\omega. \end{equation}
Then for $\tau \in X^\omega$, using centrality and orthogonality of the $p_j$, we have
\begin{equation}\begin{split}
\|e^{ih}-1_{\M^\omega}\|_{2,\tau}^2 
&= \tau\Big(\big(\sum_{j=1}^n (e^{ih_j}-1)p_j\big)^*\big(\sum_{j=1}^n (e^{ih_j}-1)p_j\big)\Big) \\
&= \tau\Big(\sum_{j=1}^n |(e^{ih_j}-1)|^2p_j \Big) \\
&\leq \frac1{\e_1^2} \sum_{j=1}^n \tau(b_jp_j) \\
&\leq \frac1{\e_1^2} 22\e_1^2\e_0^2 \\
&< 25\e_0^2.
\end{split}\end{equation}
Likewise, for $\tau \in X^\omega$,
\begin{equation}\begin{split}
&\hspace*{-2em} \big\|a^{[\e_0]}-e^{ih}|a^{[\e_0]}|\big\|_{2,\tau}^2 \\
&= \tau\Big(\big(\sum_{j=1}^n a^{[\e_0]}p_j-e^{ih_j}p_j|a^{[\e_0]}\big)^*\big(\sum_{j=1}^n a^{[\e_0]}p_j-e^{ih_j}p_j|a^{[\e_0]}\big)\Big) \\
&= \tau\Big(\sum_{j=1}^n \big|a^{[\e_0]}|-e^{ih_j}|a^{[\e_0]}|\big|^2p_j\Big) \\
&\leq \frac1{24\e_0^2} \sum_{j=1}^n \tau(h_jp_j) \\
&\leq \frac1{24\e_0^2} 22\e_1^2\e_0^2 \\
&< \e_1^2.
\end{split} \end{equation}
This proves that, $\|e^{ih}-1_{\M^\omega}\|_{2,X^\omega}<5\e_0$ and $\big\|a^{[\e_0]}-e^{ih}|a^{[\e_0]}|\big\|_{2,X^\omega}<\e_1$.
By lifting $h$ to a sequence of self-adjoints in $\M$, and using an element $k\in\M$ sufficiently far in this sequence, we therefore have
$\|e^{ik}-1\|_{2,X}\leq 5\e_0$ and  
$\big\|a^{[\e_0]}- e^{ik}|a^{[\e_0]}|\big\|_{2,X}\leq \e_1$.
The conclusion holds upon setting $u= e^{ik}$.
\end{proof}

We now prove stable rank one by showing that any element of $\M$ can be approximated arbitrarily well in the C$^*$-norm by an element with a unitary polar decomposition.
This is an iterative construction based on Proposition~\ref{prop:ControlledPolarCPoU} building a $\|\cdot\|_{2,X}$-Cauchy sequence of unitaries, which must have a $\|\cdot\|_{2,X}$-limit. The C$^*$-norm estimate $\|a-a^{[\e]}\|\leq \e$ plays a crucial role along with the the fact that closed balls in the C$^*$-norm are also $\|\cdot\|_{2,X}$-closed.

\begin{theorem}[Theorem~\ref{intro-thm:SR1}]
Let $(\M,X)$ be a factorial tracially complete $C^*$-algebra with CPoU.
Then $\M$ has stable rank one.
\end{theorem}
\begin{proof}
Let $a \in \M$. Without loss of generally assume $\|a\| \leq 1$. Let $\alpha>0$.
Set $\e_n = \frac{1}{2^n}\alpha$ for $n \in \N$. Then $\sum_{n=1}^\infty \e_n = \alpha$ and $\lim_{n\to\infty} \e_n = 0$.

Since $\M^\omega$ has stable rank one (\cite[Corollary~7.15]{TraciallyComplete}) and  unitaries in $\M^\omega$ lift to sequences of unitaries (\cite[Corollary 7.11]{TraciallyComplete}), there exists a unitary $u_0 \in \M$ such that 
\begin{equation} \big\|a - u_0|a|\big\|_{2,X}<\e_1^2. \end{equation}
Set
\begin{equation} a_1= u_0^*a, \end{equation}
and note that $|a_1|=|a|$, so that $\big\|a_1 - |a_1|\big\|_{2,X} < \e_1^2$.
By Proposition~\ref{prop:ControlledPolarCPoU}, there exists a unitary $u_1 \in \M$ such that 
\begin{equation} \big\|a_1^{[\e_1]} - u_1|a_1^{[\e_1]}|\big\|_{2,X} < \e_2^2\quad\text{and}\quad \|u_1 -1\|_{2,X}< 5\e_1. \end{equation}
Set
\begin{equation} a_2 = u_1^*a_1^{[\e_1]}, \end{equation}
so that again $|a_2|=|a_1^{[\e_1]}|$ and thus $\big\|a_2-|a_2|\big\|_{2,X}<\e_2^2$.
Continuing in this fashion, we obtain a sequence of unitaries $u_n \in \M$ such that for all $n \in \N$
\begin{equation} \|u_n-1\|_{2,X} < 5\e_n \end{equation}
and a sequence of contractions $a_n \in \M$ defined recursively by
\begin{equation} a_{n+1}= u_n^*a_n^{[\e_n]}, \end{equation}
such that for all $n \in \N$
\begin{equation} \big\|a_n-|a_n|\big\|_{2,X} < \e_n^2. \end{equation}

For all $n \in \N$, we have
\begin{equation} 
\|u_1\cdots u_nu_{n+1}-u_1\cdots u_n\|_{2,X} \leq \|u_1\cdots u_n\|  \|u_{n+1}-1\|_{2,X} < 5\e_{n+1}. 
\end{equation}
Since $\sum_{k=1}^\infty 5\e_k < \infty$, the sequence of unitaries $(u_1\cdots u_n)_{n=1}^\infty$  is $\|\cdot\|_{2,X}$-Cauchy.
As $(\M,X)$ is tracially complete, it follows that the sequence $(u_1\cdots u_n)_{n=1}^\infty$ converges in $\|\cdot\|_{2,X}$-norm to a contraction $u \in \M$.  Since multiplication of contractions is $\|\cdot\|_{2,X}$-continuous and the adjoint is $\|\cdot\|_{2,X}$-continuous, $u$ must be a unitary.

For each $n \in \N$, we have
\begin{equation} 
\begin{split}
\|a_{n+1}-a_n\|_{2,X} &\leq \|a_{n+1}-a_n^{[\e_n]}\|_{2,X}+\|a_n^{[\e_n]}-a_n\|_{2,X} \\
&\leq \|u_n^*-1\|_{2,X}\|a\| + \|a_n^{[\e_n]}-a_n\|\\
&< 5\e_n+\e_n\\
&= 6\e_n.
\end{split}
\end{equation}
As $\sum_{n=1}^\infty 6\e_n < \infty$, the sequence $(a_n)_{n=1}^\infty$ of contractions is $\|\cdot\|_{2,X}$-Cauchy and therefore converges to a contraction $a_\infty \in \M$.
Moreover, since $\big\|a_n-|a_n|\big\|_{2,X} < \e_n^2$, it follows that $a_\infty \in \M_+$ by \cite[Proposition 3.2(iii)]{TraciallyComplete}.

For each $n \in \N$, we have the following estimate
\begin{equation} \label{eqn:cstar-estimate}
\begin{split}
\|a_1-u_1\cdots u_n a_{n+1}\| &= \|a_1-u_1\cdots u_{n-1}a_n^{[\e_n]}\| \\
&\leq \|a_1-u_1\cdots u_{n-1}a_n\| + \|u_1\cdots u_{n-1}\|\|a_n^{[\e_n]} - a_n\|\\
&\leq \|a_1-u_1\cdots u_{n-1}a_n\| + \e_n.
\end{split} 
\end{equation}
Note that this estimate is valid in C$^*$-norm, not just in the uniform 2-norm.
Applying \eqref{eqn:cstar-estimate} recursively, we obtain $\|a_1-u_1\cdots u_na_{n+1}\|\leq \e_1+\cdots+\e_n<  \alpha$  for all $n \in \N$.
By \cite[Proposition 3.2(ii)]{TraciallyComplete}, the $\|\cdot\|$-closed unit ball of $\M$ is $\|\cdot\|_{2,X}$-closed. Therefore, we have
\begin{equation} 
\|a_1-ua_\infty\|\leq \alpha. 
\end{equation}
Hence,  $\|a-u_0ua_\infty\| \leq \alpha$.
Since $u_0u$ is unitary and $a_\infty$ is positive, the element $u_0u(a_\infty + \alpha 1)$ is invertible, and we have  
$\|a-u_0u(a_\infty + \alpha 1) \| \leq 2\alpha$.
Since $\alpha$ was arbitrary, this completes the proof that $\M$ has stable rank one.
\end{proof}

\section{Applications}\label{sec:applications}

In this section, we explore some applications of Theorems~\ref{intro-thm:RR0} and \ref{intro-thm:SR1}. We begin with a discussion of the trace problem for a factorial tracially complete C$^*$-algebra $(\M,X)$. This problem first appeared in the literature as \cite[Question 1.1]{TraciallyComplete} and asks whether the inclusion $X \subseteq T(\M)$ is in fact an equality. The problem was solved by the first-named author in \cite{Ev25} when $(\M, X)$ is type II$_1$ and has CPoU. Using Theorem~\ref{intro-thm:RR0}, we can provide a simplified proof of this result. 

\begin{theorem}[Theorem~\ref{intro-thm:TP}]
    Let $(\M,X)$ be a type II$_1$ factorial tracially complete C$^*$-algebra with CPoU. Then $T(\M) = X$.
\end{theorem}
\begin{proof}
    
    Let $\tau \in T(\M)$. Let $\tau_*:K_0(\M) \rightarrow \R$ be the the induced map on $K_0$.
    By \cite[Corollary 7.19]{TraciallyComplete}, we have $(K_0(\M), K_0(\M)_+) \cong (\Aff(X,\R), \Aff(X,\R_0^+))$ with the isomorphism given by the pairing of $K_0$-classes with tracial states in $X$. In particular, the $K_0$-class of the unit $1_\M$ is mapped to the constant function $1_{\Aff(X,\R)}$. 
    
    We may therefore view $\tau_*$ as an order-preserving group homomorphism $\tau_*:\Aff(X, \R) \rightarrow \R$ such that $\tau_*(1_{\Aff(X,\R)}) = 1$. As $\tau_*$ is a group homomorphism between $\R$-vector spaces, it is automatically $\mathbb{Q}$-linear. Since $\tau_*$ is order preserving and satisfies $\tau_*(1_{\Aff(X,\R)}) = 1$, it follows that $\tau_*$ is $\R$-linear and we have $|\tau_*(f)| \leq \|f\|_{\infty}$ for all $f \in \Aff(X, \R)$.

    Hence, we have shown that $\tau_*$ is a state on the complete Archimedean order unit space $\Aff(X,\R)$. By Kadison duality (\cite{Kadison51}), $\tau_*$ is given by 
    evaluation at some $\tau_0 \in X$.\footnote{Alternatively, extend $\tau_*$ to a $\C$-linear state on $C(X)$ using the Hahn--Banach theorem. This state is given by integration with respect to a Radon probability measure $\mu_0$ on $X$. Take $\tau_0$ to be the barycentre of the measure $\mu_0$.}  It follows that $\tau(p) = \tau_0(p)$ for all projections $p \in \M$. 
    Since $\M$ has real rank zero by Theorem~\ref{intro-thm:RR0}, the linear span of its projections is $\|\cdot\|$-dense in $\M$ by \cite[Proposition 14]{Ped80}. Therefore, $\tau(a) = \tau_0(a)$ for all $a \in \M$. Hence, $\tau \in X$. 
\end{proof}

We now consider consequences of Theorems~\ref{intro-thm:RR0} and~\ref{intro-thm:SR1} for the Cuntz semigroup of a type II$_1$ factorial tracially complete C$^*$-algebra $(\M,X)$; see Section~\ref{subsec:cuntz} for the relevant background on the Cuntz semigroup.

\begin{theorem}[Theorem~\ref{intro-thm:Cuntz}]
        Let $(\M,X)$ be a type II$_1$ factorial tracially complete C$^*$-algebra with CPoU. Then
    \begin{enumerate}
        \item For any $a, b \in \M_+$, $a \precsim b$ if and only for every $\e> 0$ there exists a unitary $u \in \M$ such that $u(a-\epsilon)_+u^* \in \overline{bAb}$.
        \item The Cuntz semigroup $\Cu(\M)$ is algebraic and its compact part is isomorphic to $V(\M) \cong \Aff(X,\R^+_0)$.
        \item The Cuntz semigroup $\Cu(\M)$ is almost unperforated and almost divisible, i.e.\ $\M$ is pure.  
    \end{enumerate}   
\end{theorem}
\begin{proof}
(i): By Theorem~\ref{intro-thm:SR1}, $\M$ has stable rank one. The claim now follows by \cite[Proposition 2.4(v)]{Ro92}.

(ii): By Theorem~\ref{intro-thm:RR0}, $\M$ has real rank zero. Therefore, $\Cu(\M)$ is algebraic by \cite[Corollary 5]{CEI08}; see also \cite[Section 5.5]{RFT18}. Since $\M$ is stably finite, the subsemigroup of compact elements $S_c \subseteq \Cu(\M)$ can be identified with the Murray--von Neumann semigroup $V(\M)$ using \cite[Proposition 3.5]{BrownCiuperca08} and \cite[Proposition 2.1]{Ro92}. Moreover, $V(\M) \cong \Aff(X,\R_0^+)$ by \cite[Corollary 7.19]{TraciallyComplete}.

(iii): Let $x,y \in \Cu(K)$ and $k \in \N_+$. Suppose $(k+1)x \leq ky$. By (ii), there are increasing sequences of compact elements $(x_n)_{n=1}^\infty$ and $(y_m)_{m=1}^\infty$ in $\Cu(\M)$ such that $x = \sup_n x_n$ and $y = \sup_m y_m$. Fix $n \in \N$. Then $(k+1)x_n \leq ky = \sup_m ky_m$. Since $(k+1)x_n$ is compact, there must exist an $m \in \N$ such that $(k+1)x_n \leq ky_m$. Since $\Aff(X,\R^+_0)$ is unperforated, it follows that $x_n \leq y_m$. By transitivity, $x_n \leq y$. Since $n$ was arbitrary, it follows that $x \leq y$. Therefore, $\Cu(\M)$ is almost unperforated. 
     
Let $x,x' \in \Cu(K)$ and $k \in \N_+$. Suppose $x' \ll x$. By (ii), there is an increasing sequence of compact elements $(x_n)_{n=1}^\infty$ such that $x = \sup_n x_n$. As $x' \ll x$, there exists an $n \in \N$ such that $x' \leq x_n \leq x$. As $\Aff(X,\R^+_0)$ is divisible, there exists a compact $y \in \Cu(\M)$ such that $x_n = ky$. We then have $ky \leq x$ and $x' \leq ky \leq (k+1)y$. Therefore, $\Cu(\M)$ is almost divisible.\qedhere
\end{proof}


\begin{thebibliography}{10}

\bibitem{APLT22}
R.~Antoine, F.~Perera, L.~Robert, and H.~Thiel.
\newblock $\rm{C}^*$-algebras of stable rank one and their {C}untz semigroups.
\newblock {\em Duke Math. J.}, 171(1):33--99, 2022.

\bibitem{RFT18}
R.~Antoine, F.~Perera, and H.~Thiel.
\newblock Tensor products and regularity properties of {C}untz semigroups.
\newblock {\em Mem. Amer. Math. Soc.}, 251(1199):viii+191, 2018.

\bibitem{APTVarxiv}
R.~Antoine, F.~Perera, H.~Thiel, and E.~Vilalta.
\newblock Pure {$\rm{C}^*$}-algebras.
\newblock arXiv:2406.11052, 2024.

\bibitem{BrownPed91}
L.~Brown and G.~Pedersen.
\newblock {$\rm{C}^*$}-algebras of real rank zero.
\newblock {\em J. Funct. Anal.}, 99(1):131--149, 1991.

\bibitem{BrownCiuperca08}
N.~Brown and A.~Ciuperca.
\newblock Isomorphism of {H}ilbert modules over stably finite
  {$\rm{C}^*$}-algebras.
\newblock {\em J. Funct. Anal.}, 257(1):332--339, 2009.

\bibitem{TraciallyComplete}
J.~Carrión, J.~Castillejos, S.~Evington, J.~Gabe, C.~Schafhauser, A.~Tikuisis,
  and S.~White.
\newblock Tracially complete {$\rm{C}^*$}-algebras.
\newblock arXiv:2310.20594, 2023.

\bibitem{classification1}
J.~Carrión, J.~Gabe, C.~Schafhauser, A.~Tikuisis, and S.~White.
\newblock Classifying {$^*$}-homomorphisms {I}: Unital simple nuclear
  {C$^*$}-algebras.
\newblock arXiv:2307.06480, 2023.

\bibitem{CE}
J.~Castillejos and S.~Evington.
\newblock Nuclear dimension of simple stably projectionless
  {$\rm{C}^*$}-algebras.
\newblock {\em Anal. PDE}, 13(7):2205--2240, 2020.

\bibitem{CETW-classification}
J.~Castillejos, S.~Evington, A.~Tikuisis, and S.~White.
\newblock Classifying maps into uniform tracial sequence algebras.
\newblock {\em M\"unster J. Math.}, 14(2):265--281, 2021.

\bibitem{CETW}
J.~Castillejos, S.~Evington, A.~Tikuisis, and S.~White.
\newblock Uniform property {$\Gamma$}.
\newblock {\em Int. Math. Res. Not. IMRN}, 2022(13):9864--9908, 2022.

\bibitem{CETWW}
J.~Castillejos, S.~Evington, A.~Tikuisis, S.~White, and W.~Winter.
\newblock Nuclear dimension of simple {$\rm{C}^*$}-algebras.
\newblock {\em Invent. Math.}, 224(1):245--290, 2021.

\bibitem{CEI08}
K.~Coward, G.~Elliott, and C.~Ivanescu.
\newblock The {C}untz semigroup as an invariant for {$\rm{C}^*$}-algebras.
\newblock {\em J. Reine Angew. Math.}, 623:161--193, 2008.

\bibitem{Ev18}
S.~Evington.
\newblock {\em {$\rm{W}^*$}-Bundles}.
\newblock PhD thesis, University of Glasgow, 2018.
\newblock Available online at http://theses.gla.ac.uk/8650/.

\bibitem{Ev25}
S.~Evington.
\newblock Traces on the uniform tracial completion of {$\mathcal{Z}$}-stable
  {$\rm{C}^*$}-algberas.
\newblock {\em J. London. Math. Soc.}, 111:e70207, 2025.

\bibitem{Fu26}
X.~Fu.
\newblock From stable rank one to real rank zero: A note on tracial approximate
  oscillation zero.
\newblock arXiv:2512.23911, 2025.

\bibitem{HV84}
R.~Herman and L.~Vaserstein.
\newblock The stable range of {$\rm{C}^*$}-algebras.
\newblock {\em Invent. Math.}, 77(3):553--555, 1984.

\bibitem{Kadison51}
R.~Kadison.
\newblock A representation theory for commutative topological algebra.
\newblock {\em Mem. Amer. Math. Soc.}, 7:39, 1951.

\bibitem{KLTV}
G.~Kopsacheilis, H.-C. Liao, A.~Tikuisis, and A.~Vaccaro.
\newblock Uniform property {$\Gamma$} and the small boundary property.
\newblock {\em Trans. Amer. Math. Soc.}, 379(1):487--509, 2026.

\bibitem{MS12}
H.~Matui and Y.~Sato.
\newblock Strict comparison and {$\mathcal{Z}$}-absorption of nuclear
  {$\rm{C}^*$}-algebras.
\newblock {\em Acta Math.}, 209(1):179--196, 2012.

\bibitem{MS14}
H.~Matui and Y.~Sato.
\newblock Decomposition rank of {UHF}-absorbing {$\rm{C}^*$}-algebras.
\newblock {\em Duke Math. J.}, 163(14):2687--2708, 2014.

\bibitem{Oz13}
N.~Ozawa.
\newblock Dixmier approximation and symmetric amenability for {$\rm
  C^*$}-algebras.
\newblock {\em J. Math. Sci. Univ. Tokyo}, 20(3):349--374, 2013.

\bibitem{Ped80}
G.~Pedersen.
\newblock The linear span of projections in simple {$\rm{C}^*$}-algebras.
\newblock {\em J. Operator Theory}, 4(2):289--296, 1980.

\bibitem{Rieffel83}
M.~Rieffel.
\newblock Dimension and stable rank in the {$K$}-theory of
  {$\rm{C}^*$}-algebras.
\newblock {\em Proc. London Math. Soc. (3)}, 46(2):301--333, 1983.

\bibitem{Ro92}
M.~R{\o}rdam.
\newblock On the structure of simple {$\rm{C}^*$}-algebras tensored with a
  {UHF}-algebra. {II}.
\newblock {\em J. Funct. Anal.}, 107(2):255--269, 1992.

\bibitem{Ro04}
M.~R{\o}rdam.
\newblock The stable and the real rank of {$\mathcal{Z}$}-absorbing
  {$\rm{C}^*$}-algebras.
\newblock {\em Int. J. Math.}, 15(10):1065--1084, 2004.

\bibitem{SW25}
G.~Szab\'o and L.~Wouters.
\newblock Equivariant property gamma and the tracial local-to-global principle
  for {${\rm C}^*$}-dynamics.
\newblock {\em Anal. PDE}, 18(6):1385--1432, 2025.

\bibitem{Takesaki3}
M.~Takesaki.
\newblock {\em Theory of operator algebras. {III}}, volume 127 of {\em
  Encyclopaedia of Mathematical Sciences}.
\newblock Springer-Verlag, Berlin, 2003.

\bibitem{Wi12}
W.~Winter.
\newblock Nuclear dimension and {$\mathcal{Z}$}-stability of pure
  {$\rm{C}^*$}-algebras.
\newblock {\em Invent. Math.}, 187(2):259--342, 2012.

\bibitem{Zhang90}
S.~Zhang.
\newblock A property of purely infinite simple {$\rm{C}^*$}-algebras.
\newblock {\em Proc. Amer. Math. Soc.}, 109(3):717--720, 1990.

\end{thebibliography}
\end{document}